\documentclass[12pt]{amsart}

\usepackage{amsmath,a4wide}
\usepackage{amssymb,amsthm,bbm}
\usepackage[usenames,dvipsnames]{color}
\usepackage{hyperref}
\usepackage{cite}

\newtheorem{definition}{Definition}
\newtheorem{lemma}{Lemma}
\newtheorem{theorem}{Theorem}
\newtheorem{corollary}{Corollary}
\theoremstyle{definition}
\newtheorem{remark}{Remark}
\newtheorem*{ackno}{Acknowledgement}

\DeclareMathOperator{\xctint}{I}
\DeclareMathOperator{\numint}{Q}
\DeclareMathOperator{\WCE}{wce}
\newcommand{\eps}{\varepsilon}
\newcommand{\dd}{\mathrm{d}}
\begin{document}
\title[Hyperuniformity]
{Hyperuniform point sets on the sphere: deterministic aspects}

\author[J. S. Brauchart]{Johann S. Brauchart\textsuperscript{\dag}}
\address[J. B., P. G., W. K.]{Institute of Analysis and Number Theory,
  Graz University of Technology,
  Kopernikusgasse 24.
8010 Graz,
Austria}
\curraddr[W. K.]{Department of Mathematics, Vanderbilt University,
1326 Stevenson Center, Nashville, TN 37240, USA}
\email[J. B.]{j.brauchart@tugraz.at}
\thanks{\dag{} This author is supported by the Lise Meitner scholarship M~2030
of the Austrian Science Foundation FWF}

\author[P. J. Grabner]{Peter J. Grabner\textsuperscript{\ddag,*}}
\thanks{\textsuperscript{\ddag} These authors are supported by the Austrian
  Science Fund FWF project F5503 (part of the Special Research Program (SFB)
  ``Quasi-Monte Carlo Methods: Theory and Applications'')}
\thanks{\textsuperscript{*} This author is supported by  by the Doctoral
  Program ``Discrete Mathematics'' W1230.}
\email[P. G.]{peter.grabner@tugraz.at}

\author[W. Kusner]{W\"oden Kusner\textsuperscript{\ddag}}
\email[W. K.]{wkusner@gmail.com}

\dedicatory{Dedicated to Robert F. Tichy on the occasion of his
  60\textsuperscript{th} birthday.}
\maketitle

\begin{abstract}
  The notion of hyperuniformity originally introduced as a measure of
  regularity of infinite point sets in Euclidean space is generalised
  and extended to sequences of finite point sets on the sphere. It is
  shown that hyperuniformity implies uniform
  distribution. Furthermore, it is shown that QMC-design sequences
  with strength at least $\frac{d+1}2$ and especially sequences of
  spherical designs of optimal growth order are hyperuniform.
\end{abstract}
\section{Introduction}\label{sec:intro}
Hyperuniformity was introduced by Torquato and Stillinger
\cite{torquato2003local}
(cf.~\cite{Pietronero-Gabrielli-Labini2002:statistical_physics}) to describe
idealised \emph{infinite} point configurations, which exhibit properties
between order and disorder. Such configurations $X$ occur as jammed packings,
in colloidal suspensions, as well as quasi-crystals. The main feature of
hyperuniformity is the fact that local density fluctuations are of smaller
order than for a random (``Poissonian'') point configuration. Alternatively,
hyperuniformity can be characterised in terms of the structure factor
\begin{equation*}
  S(\mathbf{k})=\lim_{B\to\mathbb{R}^d}\frac1{\#(B\cap X)}
\sum_{\mathbf{x},\mathbf{y}\in B\cap X}e^{i\langle\mathbf
  k,\mathbf{x}-\mathbf{y}\rangle}\quad
\text{(thermodynamic limit)}
\end{equation*}
by $\lim_{\mathbf{k}\to\mathbf{0}}S(\mathbf{k})=0$. This
\emph{thermodynamic limit} is understood in the sense that the volume
$B$ (for instance a ball of radius $R$) tends to the whole space
$\mathbb{R}^d$ while $\lim_{B\to\mathbb{R}^d}\frac{\#(B\cap
  X)}{\mathrm{vol}(B)}=\rho$, the density.

 For a long time in the physics literature it has been observed that there are
 large (ideally infinite) particle systems that exhibit structural behaviour
 between crystalline order and total disorder. Very prominent examples are
 given by quasi-crystals and jammed sphere packings.  The discovery of such
 physical materials which lie between crystalline order and disordered
 materials has initiated research in physics as well as in mathematics. We just
 mention de~Bruijn's Fourier analytic explanation for the diffraction pattern
 of quasi-crystals \cite{dBr86} and the extensive collection of articles on
 quasi-crystals \cite{AxGr95}.

 After the introduction of hyperuniformity in \cite{torquato2003local} as a
 concept to measure the occurrence of ``intermediate'' order as for
 quasi-crystals or jammed packings the notion has developed
 tremendously. Hyperuniformity has found applications far beyond physics:
 colour receptors in bird eyes exhibit hyperuniform structure
 \cite{Jiao-Lau-Hatzikirou+2014:avian_photoreceptor}, as do the keratin
 nanostructures in bird feathers \cite{Liew-Forster-Noh+2011:short-range}, as
 do energy minimising point configurations, and of course quasi-crystals
 \cite{Oguz_Socolar_Steinhardt+2017:hyperuniformity_quasicrystals}.

 The basic framework is as follows. Let $X$ be a countable discrete subset of
 $\mathbb{R}^d$ and $\Omega\subset\mathbb{R}^d$ be a test set (``window''), in
 most cases the unit ball. Then
 $N_{\mathbf{x}+t\Omega}=\#((\mathbf{x}+t\Omega)\cap X)$ counts the number of
 points in the translated and dilated copy of $\Omega$. As a general assumption
 we take that $X$ has a density $\rho$, meaning that
\begin{equation*}
  N_{\mathbf{x}+t\Omega}\sim \rho t^d\mathrm{vol}(\Omega)
\end{equation*}
for $t\to\infty$, independent of $\mathbf{x}$. Based on this assumption, the
thermodynamic limit can be taken to define the expectation  of
$N_{\mathbf{x}+t\Omega}$ as
\begin{equation*}
  \langle N_{t\Omega}\rangle=
\lim_{R\to\infty}\frac1{\mathrm{vol}(B(\mathbf{0},R))}\int_{B(\mathbf{0},R))}
N_{\mathbf{x}+t\Omega}\,\dd\mathbf{x},
\end{equation*}
where $B(\mathbf{0},R)$ denotes the ball of radius $R$ around $\mathbf{0}$.
For a random (``Poissonian'') point pattern the variance satisfies
$\langle N_{t\Omega}^2\rangle-\langle N_{t\Omega}\rangle^2= \langle
N_{t\Omega}\rangle=\rho t^d\mathrm{vol}(\Omega)$.  For point sets like
quasi-crystals or jammed packings the behaviour is different: the
variance has smaller order of magnitude as $t\to\infty$, ideally
\begin{equation}\label{eq:strong-hyper}
  \langle N_{t\Omega}^2\rangle-\langle
  N_{t\Omega}\rangle^2=\mathcal{O}(t^{d-1})\asymp
\mathrm{surface}(t\Omega).
\end{equation}
Here and throughout, we use the notation $f(x)\asymp g(x)$ for
$f(x)=\mathcal{O}(g(x))$ and $g(x)=\mathcal{O}(f(x))$ for the indicated range
of $x$.

Such behaviour is clearly displayed by lattices and randomly distorted
lattices, and some quasi-crystals (depending on Diophantine properties
of their construction parameters)
\cite{Oguz_Socolar_Steinhardt+2017:hyperuniformity_quasicrystals}. There
is numerical evidence that jammed sphere packings
\cite{Kansal_Torquato_Stillinger2002:diversity_jammed,
  Torquato_Truskett_Debenedetti2000:is_random_close} also exhibit such
behaviour. More generally, a point set is called hyperuniform if
\begin{equation*}
   \langle N_{t\Omega}^2\rangle-\langle
  N_{t\Omega}\rangle^2=o(t^d);
\end{equation*}
it is called strongly hyperuniform if \eqref{eq:strong-hyper} holds. It has
been shown in \cite{torquato2003local} that \eqref{eq:strong-hyper} is the
best possible order that can occur.

\section{Hyperuniformity on the sphere}\label{sec:sphere}
Complementing the extensive study of the notion of hyperuniformity in the
infinite setting, we are interested in studying an analogous property of
sequences of point sets in compact spaces. For convenience we study the
$d$-dimensional sphere $\mathbb{S}^d$. Our ideas immediately generalise to
homogeneous spaces; further generalisations might be more elaborate,
since we rely heavily on harmonic analysis and specific properties of special
functions. Throughout this paper $\sigma=\sigma_d$ will denote the normalised
surface area measure on $\mathbb{S}^d$. We suppress the dependence on $d$ in
this notation.

In order to adapt to the compact setting, we replace the infinite set $X$
studied in the classical notion of hyperuniformity by a sequence of finite
point sets $(X_N)_{N\in A}$, where we assume that $\#X_N=N$. By using an
infinite set $A\subseteq\mathbb{N}$ as index set, we always allow for
subsequences. Furthermore, the set
$X_N=\{\mathbf{x}_1^{(N)},\ldots,\mathbf{x}_N^{(N)}\}$ consists of points
depending on $N$; for the ease of notation, we omit this dependence throughout
the paper.  We propose the Definition~\ref{def-hyper} below, which we study in
further detail in this paper.

Throughout the paper we use the notation
\begin{equation*}
  C(\mathbf{x},\phi)=\{\mathbf{y}\in\mathbb{S}^d\mid
  \langle\mathbf{x},\mathbf{y}\rangle>\cos\phi\}
\end{equation*}
for the spherical cap of opening angle $\phi$ with center $\mathbf{x}$. The
normalised surface area of the cap is given by
\begin{equation} \label{eq:normalised.surface.cap}
  \sigma\left(C(\mathbf{x},\phi)\right)=
  \gamma_d\int_0^\phi\sin(\theta)^{d-1}\dd\theta\asymp\phi^{d} \quad\text{as
  }\phi\to0,
\end{equation}
where
\begin{equation*}
  \gamma_d=\left(\int_0^\pi\sin(\theta)^{d-1}\dd\theta\right)^{-1}=
  \frac{\Gamma(d)}{2^{d-1}\Gamma(d/2)^2}.
\end{equation*}
Notice that $\gamma_d=\frac{\omega_{d-1}}{\omega_d}$, where $\omega_d$ is the
surface area of $\mathbb{S}^d$.

For the reader's convenience and for later reference we first recapitulate the
definition of uniform distribution of a sequence $(X_N)_{N\in A}$ of point sets
on the sphere $\mathbb{S}^d$ (see
\cite{Kuipers_Niederreiter1974:uniform_distribution_sequences,
  Drmota_Tichy1997:sequences_discrepancies} as general references on the theory
of uniform distribution).
\begin{definition}[Uniform distribution]\label{def-ud}
  A sequence of point sets $(X_N)_{N\in A}$ is called uniformly distributed on
  $\mathbb{S}^d$, if for all caps $C(\mathbf{x},\phi)$
  \textup{(}$\mathbf{x}\in\mathbb{S}^d$ and $\phi\in[0,\pi]$\textup{)} the relation
  \begin{equation}
    \label{eq:ud}
    \lim_{N\to\infty}\frac1N\sum_{i=1}^N
    \mathbbm{1}_{C(\mathbf{x},\phi)}(\mathbf{x}_i)=
    \sigma(C(\mathbf{x},\phi))
  \end{equation}
  holds. Here $\mathbbm{1}_C$ denotes the indicator function of the set $C$.
\end{definition}
It is known from the general theory of uniform distribution (see
\cite{Kuipers_Niederreiter1974:uniform_distribution_sequences}) that
\eqref{eq:ud} is equivalent to
\begin{equation}
  \label{eq:ud-p}
  \lim_{N\to\infty}\frac1{N^2}\sum_{i,j=1}^N
  P_n^{(d)}(\langle\mathbf{x}_i,\mathbf{x}_j\rangle)=0\quad\text{ for }n\geq1,
\end{equation}
where $P_n^{(d)}(x)$ is the $n$-th (generalised) Legendre polynomial normalised
by $P_n^{(d)}(1)=1$. These functions are the zonal spherical harmonics on
$\mathbb{S}^d$ (see \cite{Mueller1966:spherical_harmonics}). Notice that
\begin{equation*}
  Z(d,n)P_n^{(d)}(x)=\frac{n+\lambda}{\lambda}C_n^\lambda(x),
\end{equation*}
where $C_n^\lambda(x)$ is the $n$-th Gegenbauer polynomial with index
$\lambda=\frac{d-1}2$ (see
\cite{Magnus_Oberhettinger_Soni1966:formulas_theorems}). We write
$Z(d,n)=\frac{2n+d-1}{d-1}\binom{n+d-2}{d-2}$ for the dimension of the
space of spherical harmonics of degree $n$ on $\mathbb{S}^d$.

The spherical cap discrepancy
\begin{equation*}
  D_N^\infty(X_N)=\sup_{\mathbf{x},\phi}
  \left|\sum_{n=1}^N\mathbbm{1}_{C(\mathbf{x},\phi)}(\mathbf{x}_n)-
    N\sigma(C(\cdot,\phi))\right|
\end{equation*}
provides a well studied quantitative measure of uniform distribution (see
\cite{Kuipers_Niederreiter1974:uniform_distribution_sequences,
  Beck_Chen1987:irregularities_distribution}). Uniform distribution of
$(X_N)_{N\in A}$ is equivalent to
\begin{equation*}
  \lim_{N\to\infty}\frac1ND_N^\infty(X_N) = 0.
\end{equation*}
In this paper we will study the \emph{number variance}.
  \begin{definition}[Number variance]\label{def:number-v}
  Let $(X_N)_{N\in\mathbb{N}}$ be a sequence of point sets on the sphere
  $\mathbb{S}^d$. The \emph{number variance} of the sequence for caps of
  opening angle $\phi$ is given by
  \begin{equation*}
    V(X_N,\phi):=\mathbb{V}_{\mathbf{x}}\#\left(X_N\cap
      C(\mathbf{x},\phi)\right)=\int_{\mathbb{S}^d}
  \left(\sum_{n=1}^N\mathbbm{1}_{C(\mathbf{x},\phi)}(\mathbf{x}_n)-
    N\sigma(C(\cdot,\phi))\right)^2\,\dd\sigma(\mathbf{x}).
  \end{equation*}
  \end{definition}

This quantity appears in the classical measure of uniform
  distribution given by the $L^2$-discrepancy
\begin{equation}
  \label{eq:discr-L2}
  D_N^2(X_N)= \left(\int_0^\pi V(X_N,\phi)
\sin(\phi)\,\dd\phi\right)^{\frac12},
\end{equation}
where uniform distribution of $(X_N)_{N\in A}$ is also equivalent to
\begin{equation*}
  \lim_{N\to\infty}\frac1ND_N^2(X_N)=0.
\end{equation*}

As in the Euclidean case we define hyperuniformity by a comparison
between the behaviour of the number variance of a sequence of point
sets and the i.i.d case. For i.i.d points the variance is
$N\sigma(C(\cdot,\phi))(1-\sigma(C(\cdot,\phi)))$, which has order of
magnitude $N$, $N\sigma(C(\cdot,\phi_N))$, and $t^d$, respectively, in
the three cases \eqref{eq:large}, \eqref{eq:small}, and
\eqref{eq:threshold} listed below.

\begin{definition}[Hyperuniformity]\label{def-hyper}
  Let $(X_N)_{N\in\mathbb{N}}$ be a sequence of point sets on the sphere
  $\mathbb{S}^d$.
  A sequence is called
  \begin{itemize}
  \item \textbf{hyperuniform for large caps} if
    \begin{equation}
      \label{eq:large}
      V(X_N,\phi)=o\left(N\right)\quad \text{as } N\to\infty
    \end{equation}
 for all $\phi\in(0,\frac\pi2)$ ;
\item \textbf{hyperuniform for small caps} if 
\begin{equation}
      \label{eq:small}
      V(X_N,\phi_N)=o\left(N\sigma(C(\cdot,\phi_N))\right)
      \quad \text{as } N\to\infty
    \end{equation}
      and all sequences $(\phi_N)_{N\in\mathbb{N}}$ such that
  \begin{enumerate}
  \item $\lim_{N\to\infty}\phi_N=0$
  \item $\lim_{N\to\infty}N\sigma(C(\cdot,\phi_N))=\infty$,
      which is equivalent to $\phi_NN^{\frac1d}\to\infty$.
  \end{enumerate}
\item \textbf{hyperuniform for caps at threshold order}, if
  \begin{equation}
    \label{eq:threshold}
    \limsup_{N\to\infty}V(X_N,tN^{-\frac1d})=
  \mathcal{O}(t^{d-1}) \quad\text{as } t\to\infty.
  \end{equation}
\end{itemize}
\end{definition}
\begin{remark}
  The case analogous to the Euclidean definition is the third case:
  hyperuniform for caps at threshold order. The limit
  $N\to\infty$ is the analogue of the thermodynamic limit by rescaling
  to a sphere of radius $N^{\frac1d}$.
\end{remark}
In order to determine further properties of hyperuniform sequences of sets, we
derive an alternative expression for the number variance $V(X_N,\phi_N)$. We
refer to \cite{Mueller1966:spherical_harmonics} as a general reference for
spherical harmonics in arbitrary dimension, and to
\cite{Andrews-Askey-Roy1999:Special_functions,
  Magnus_Oberhettinger_Soni1966:formulas_theorems} as references for the
various formulas and relations between special functions, especially orthogonal
polynomials.

Recall the Laplace series for the indicator function of the spherical cap
$C(\mathbf{x},\phi)$:
\begin{equation*}
  \mathbbm{1}_{C(\mathbf{x},\phi)}(\mathbf{y})=
  \sigma(C(\cdot,\phi))+\sum_{n=1}^\infty a_n(\phi)Z(d,n)
  P_n^{(d)}(\langle\mathbf{x},\mathbf{y}\rangle),
\end{equation*}
where the Laplace coefficients are given by
\begin{equation}\label{eq:an-phi}
  a_n(\phi)=\gamma_d\int_0^\phi
  P_n^{(d)}(\cos(\theta))\sin(\theta)^{d-1}\,\dd\theta= 
  \frac{\gamma_d}d\sin(\phi)^dP_{n-1}^{(d+2)}(\cos(\phi)),\quad n\geq1.
\end{equation}

The variance $V(X_N,\phi)$ can be expressed formally as
\begin{equation}\label{eq:var-harmonic}
\begin{aligned}
  V(X_N,\phi)&=\int_{\mathbb{S}^d}
  \left(\sum_{i=1}^N\mathbbm{1}_{C(\mathbf{x}_i,\phi)}(\mathbf{x})-
    N\sigma(C(\cdot,\phi))\right)^2\,\dd\sigma(\mathbf{x})\\
&=\sum_{i,j=1}^N\sum_{n=1}^\infty a_n(\phi)^2Z(d,n)
  P_n^{(d)}(\langle\mathbf{x}_i,\mathbf{x}_j\rangle),
\end{aligned}
\end{equation}
by interpreting the integral as a (spherical) convolution. This
follows from the Funk-Hecke formula
\begin{equation*}
  \int_{\mathbb{S}^d}P_m^{(d)}(\langle\mathbf{x},\mathbf{y}\rangle)
P_n^{(d)}(\langle\mathbf{y},\mathbf{z}\rangle)\,\dd\sigma(\mathbf{y})=
\delta_{m,n}P_n^{d}(\langle\mathbf{x},\mathbf{z}\rangle).
\end{equation*}
We also remark here that 
\begin{equation}
  \label{eq:pos-def}
  \sum_{i,j=1}^N  P_n^{(d)}(\langle\mathbf{x}_i,\mathbf{x}_j\rangle)\geq0
\end{equation}
by the positive definiteness of $P_n^{(d)}$ (see
\cite{Schoenberg1938:positive_definite}).

Notice that the function
\begin{equation}\label{eq:g_phi}
  g_\phi(x)=\sum_{n=1}^\infty a_n(\phi)^2Z(d,n)P_n^{(d)}(x),\quad -1\leq x\leq1,
\end{equation}
is positive definite in the sense of Schoenberg
\cite{Schoenberg1938:positive_definite}. Furthermore, the estimate
\begin{equation}
  \label{eq:Gegenbauer}
  \left|P_n^{(d)}(\cos(\phi))\right|\leq
  \min\left(1,\frac{c_d}{(n\sin(\phi))^{\frac{d-1}2}}\right)
\end{equation}
holds for a constant $c_d$ depending only on the dimension $d$ (see
\cite{Kogbetliantz24,Szegoe1975:orthogonal_polynomials}).
This gives the estimate
\begin{equation*}
  a_n(\phi)^2=\mathcal{O}\left(\frac{\sin(\phi)^{d-1}}{n^{d+1}}\right),
\end{equation*}
which holds uniformly for $\phi\in[0,\pi]$. This together with
$Z(d,n)=\mathcal{O}(n^{d-1})$ shows absolute and uniform convergence of the
series \eqref{eq:g_phi} and thus \eqref{eq:var-harmonic}.
\subsection{Hyperuniformity for large caps}\label{sec:large}
\begin{theorem} \label{thm:limit.hyperuniformity.large.caps}
  Let $(X_N)_{N\in\mathbb{N}}$ be a sequence of point sets, which is
  hyperuniform for large caps. Then for all $n\geq1$
  \begin{equation}
    \label{eq:limit}
    \lim_{N\to\infty}\frac1N\sum_{i,j=1}^N
    P_n^{(d)}(\langle\mathbf{x}_i,\mathbf{x}_j\rangle)=0.
  \end{equation}
  As a consequence, sequences which are hyperuniform for large caps are
  uniformly distributed.
\end{theorem}
\begin{proof}
  Assume that $(X_N)_{N\in\mathbb{N}}$ is hyperuniform for large caps. Then
  inserting the definition into \eqref{eq:var-harmonic} gives
  \begin{equation*}
    0=\lim_{N\to\infty}\frac{V(X_N,\phi)}N\geq Z(d,n)a_n(\phi)^2
    \limsup_{N\to\infty}
    \frac1N\sum_{i,j=1}^N P_n^{(d)}(\langle\mathbf{x}_i,\mathbf{x}_j\rangle)
  \end{equation*}
  for every $n$ and every $\phi\in(0,\frac\pi2)$, which implies
  \eqref{eq:limit} by the positive definiteness of Legendre
  polynomials \eqref{eq:pos-def}, positivity of the Laplace
  coefficients of the series and uniform convergence.
\end{proof}
\begin{remark}\label{rem1a}
  Of course, the uniform distribution of hyperuniform point sets is no
  surprise, since the uniform density of points was built into the computation
  of variance. Furthermore, all caps of a fixed size are used in the definition
  of this regime of hyperuniformity, similarly to the definition of uniform
  distribution. The extra convergence order in \eqref{eq:limit} is the key
  observation in Theorem~\ref{thm:limit.hyperuniformity.large.caps}. Similar
  phenomena will occur in Section~\ref{sec:QMC-designs}, where the notion of
  a QMC-design as defined in \cite{Brauchart_Saff_Sloan+2014:qmc-designs} is
  exploited further and hyperuniformity of QMC-designs is shown.
\end{remark}
\begin{remark}\label{rem1}
  Notice that  it does not suffice to assume that
  \eqref{eq:large} holds for only one value of $\phi\in(0,\frac\pi2)$. For
  values of $\phi$ for which one of the coefficients $a_{n_0}(\phi)$ vanishes,
  nothing can be said about the limit \eqref{eq:limit} for $n=n_0$. There are
  of course only countably many such values of $\phi$. Furthermore, it has been
  conjectured by T.~J.~Stieltjes
  \cite{Baillaud-Bouget1905:Correspondance-Stieltjes} that the (classical)
  Legendre polynomials $P_{2n}(x)$ and $P_{2n+1}(x)/x$ are irreducible. An
  extension of this still unproved conjecture to higher dimensional Legendre
  polynomials would imply that at most one coefficient $a_n(\phi)$ could vanish
  for a given value of $\phi\in(0,\frac\pi2)$.
\end{remark}
\begin{proof}
  We construct a point sets such that \eqref{eq:limit} holds for all
  $n\neq n_0$ and
  \begin{equation}\label{eq:lim-infty}
    \lim_{N\to\infty}\frac1N\sum_{i,j=1}^N P_{n_0}^{(d)}(\langle
    \mathbf{x}_i,\mathbf{x}_j\rangle)=\infty.
  \end{equation}
  Take a non-zero spherical harmonic function $f$ of order $n_0$ which has all
  values less than $1$ in modulus. Then
  $\dd\mu(\mathbf{x})=(1+f(\mathbf{x}))\,\dd\sigma(\mathbf{x})$ is a positive
  measure on $\mathbb{S}^d$. Then by a result of Seymour and Zaslavsky
  \cite{Seymour_Zaslavsky1984:averaging_designs} for every $t$ there exists an
  $N(t)$ such that for every $N\geq N(t)$ there is a point set $X_N$ such that
\begin{equation*}
  \frac1N\sum_{i=1}^Np(\mathbf{x}_i)=\int_{\mathbb{S}^d}p(\mathbf{x})
  \,\dd\mu(\mathbf{x})=\langle 1+f,p\rangle_{L^2(\mathbb{S}^d)}
\end{equation*}
for all spherical harmonics $p$ of degree $\leq t$. Let now
$p(\mathbf{x})=P_n^{(d)}(\langle\mathbf{x},\mathbf{y}\rangle)$ for fixed
$\mathbf{y}\in\mathbb{S}^d$. For all $n\neq n_0$ and
$1\leq n\leq t$ we have
\begin{equation*}
  \frac1N\sum_{i=1}^NP_n^{(d)}(\langle\mathbf{x}_i,\mathbf{y}\rangle)=0\quad\text{for
    every }\mathbf{y}\in\mathbb{S}^d,
\end{equation*}
from which we conclude
\begin{equation*}
  \frac1N\sum_{i,j=1}^NP_n^{(d)}(\langle\mathbf{x}_i,\mathbf{x}_j\rangle)=0.
\end{equation*}
This gives the desired limit relation.

For $n=n_0$ and $t\geq n_0$ we have
\begin{equation*}
  \frac1N\sum_{i=1}^NP_{n_0}^{(d)}(\langle\mathbf{x}_i,\mathbf{y}\rangle)=
  \int_{\mathbb{S}^d}f(\mathbf{x})P_{n_0}^{(d)}(\langle\mathbf{x},\mathbf{y}\rangle)\,
  \dd\sigma(\mathbf{x})
  \quad\text{for every }\mathbf{y}\in\mathbb{S}^d.
\end{equation*}
Taking $\mathbf{y}=\mathbf{x}_j$ and summing again yields
\begin{equation*}
  \frac1{N^2}\sum_{i,j=1}^NP_{n_0}^{(d)}(\langle\mathbf{x}_i,\mathbf{x}_j\rangle)=
  \int_{\mathbb{S}^d} \int_{\mathbb{S}^d}
  f(\mathbf{x})P_{n_0}^{(d)}(\langle\mathbf{x},\mathbf{y}\rangle)f(\mathbf{y})
  \,\dd\sigma(\mathbf{x})\,\dd\sigma(\mathbf{y})=
  \frac{\|f\|_{L^2(\mathbb{S}^d)}^2}{Z(d,n_0)}\neq0,
\end{equation*}
  which implies \eqref{eq:lim-infty}.
\end{proof}

\subsection{Hyperuniformity for small caps}\label{sec:small}
Using the definition of hyperuniformity together with \eqref{eq:var-harmonic}
we have
\begin{equation}\label{eq:hyper-harmonic}
  \frac{V(X_N,\phi_N)}{N\sigma(C(\cdot,\phi_N))}=
  \sum_{n=1}^\infty Z(d,n)\frac{a_n(\phi_N)^2}{\sigma(C(\cdot,\phi_N))}
  \frac1N\sum_{i,j=1}^N P_n^{(d)}(\langle\mathbf{x}_i,\mathbf{x}_j\rangle)\to0.
\end{equation}
By \eqref{eq:an-phi} the $n$-th Laplace coefficient of
\eqref{eq:hyper-harmonic} coefficient behaves like
\begin{equation*}
  \frac{a_n(\phi_N)^2}{\sigma(C(\cdot,\phi_N))}=
  \left(\frac{\gamma_d}dP_{n-1}^{(d+2)}(\cos\phi_N)\right)^2
  \frac{\sin(\phi_N)^{2d}}{\sigma(C(\cdot,\phi_N))}\asymp \phi_N^d
\end{equation*}
for $\phi_N\to0$ as assumed. Since $\phi_N$ is allowed to tend to $0$
arbitrarily slowly and all coefficients in \eqref{eq:hyper-harmonic} are
positive, this implies that
\begin{equation*}
  \limsup_{N\to\infty}\frac1N\sum_{i,j=1}^N
  P_n^{(d)}(\langle\mathbf{x}_i,\mathbf{x}_j\rangle)<\infty
\end{equation*}
for all $n\geq1$.

Using the fact that \eqref{eq:ud-p} is equivalent to uniform distribution of
$(X_N)_{N\in A}$ we have proved
\begin{theorem}
  Let $(X_N)_{N\in\mathbb{N}}$ be a sequence of point sets on the sphere which
  is hyperuniform for small caps. Then $(X_N)_{N\in\mathbb{N}}$ is
  asymptotically uniformly distributed.
\end{theorem}

Motivated by the analogous definition in the Euclidean case, we call the
function
\begin{equation*}
  s(n)=\lim_{N\to\infty}\frac1N\sum_{i,j=1}^N
  P_n^{(d)}(\langle\mathbf{x}_i,\mathbf{x}_j\rangle)
\end{equation*}
the \emph{spherical structure factor}, if the limit exists for all
$n\geq1$. Notice that by a diagonal argument, we can always achieve that all
such limits exist along some subsequence.
\begin{remark}\label{rem3a}
  As opposed to the case of hyperuniformity for large caps discussed in
  Remark~\ref{rem1a}, in the case of small caps the conclusion of uniform
  distribution is not directly obvious, because only ``small'' caps in the
  sense of \eqref{eq:small} are tested for the definition of uniform
  distribution.
\end{remark}
\subsection{Hyperuniformity for caps of threshold order}
\label{sec:threshold}
\begin{theorem}
  Let $(X_N)_{N\in\mathbb{N}}$ be a sequence of point sets on the sphere which
  is hyperuniform for caps of threshold order. Then $(X_N)_{N\in\mathbb{N}}$ is
  asymptotically uniformly distributed.
\end{theorem}
\begin{proof}
  We insert the definition of hyperuniformity for caps of threshold order into
  \eqref{eq:var-harmonic} to obtain
  \begin{equation*}
    V(X_N,tN^{-\frac1d})\geq a_n\left(tN^{-\frac1d}\right)^2Z(d,n)
    \sum_{i,j=1}^NP_n^{(d)}\left(\langle\mathbf{x}_i,\mathbf{x}_j\rangle\right).
  \end{equation*}
  Then \eqref{eq:an-phi} yields
  \begin{equation*}
    a_n\left(tN^{-\frac1d}\right)^2\sim\left(\frac{\gamma_d}d\right)^2t^{2d}N^{-2}
  \end{equation*}
for fixed $n\geq1$ and fixed $t>0$ and $N\to\infty$. Now by definition
\eqref{eq:threshold} we
have
\begin{equation*}
  \left(\frac{\gamma_d}d\right)^2t^{2d}Z(d,n)\limsup_{N\to\infty}\frac1{N^2}
  \sum_{i,j=1}^NP_n^{(d)}\left(\langle\mathbf{x}_i,\mathbf{x}_j\rangle\right)\leq
  \limsup_{N\to\infty}V(X_N,tN^{-\frac1d})=\mathcal{O}(t^{d-1}).  
\end{equation*}
This relation can only hold if
\begin{equation*}
  \limsup_{N\to\infty}\frac1{N^2}
  \sum_{i,j=1}^NP_n^{(d)}\left(\langle\mathbf{x}_i,\mathbf{x}_j\rangle\right)=0
\end{equation*}
for all $n\geq1$, which implies uniform distribution of the sequence $(X_N)_N$.
\end{proof}
\begin{remark}\label{rem3b}
  Similarly to the case of hyperuniformity for small caps the conclusion of
  uniform distribution of sequences of hyperuniform points sets for caps at
  threshold order is not obvious.
\end{remark}

\section{Hyperuniformity of QMC design sequences}\label{sec:QMC-designs}

A Quasi-Monte Carlo (QMC) method is an equal weight numerical
integration formula that, in contrast to Monte Carlo methods,
approximates the exact integral $\xctint(f)$ of a given continuous
real function $f$ on $\mathbb{S}^d$ using a \emph{deterministic} node
set $X_N = \{ \mathbf{x}_1, \dots, \mathbf{x}_N \} \subset
\mathbb{S}^d$:
\begin{equation*}
\xctint(f) := \int_{\mathbb{S}^d} f( \mathbf{x} ) \dd \sigma_d(
\mathbf{x} ) 
\approx \frac{1}{N} \sum_{k=1}^N f(\mathbf{x}_k) =: \numint[X_N](f).
\end{equation*}
The node set $X_N$ is chosen in a sensible way so as to guarantee
``small'' worst-case error of numerical integration,
\begin{equation*}
\WCE(\numint[X_N]; \mathbb{H}^s( \mathbb{S}^d ) ) := 
\sup \left\{ \big| \numint[X_N](f) - \xctint(f) 
\big| : f \in \mathbb{H}^s( \mathbb{S}^d ), \| f \|_{\mathbb{H}^s}
\leq 1 \right\}
\end{equation*}
with respect to a Sobolev space $\mathbb{H}^s(\mathbb{S}^d)$ over
$\mathbb{S}^d$ with smoothness index $s > \frac{d}{2}$.

Motivated by certain estimates for the worst-case error, the concept
of QMC design sequences was introduced in
\cite{Brauchart_Saff_Sloan+2014:qmc-designs}.
In the following we assume that $A$ is an infinite subset
of~$\mathbb{N}_0$.
Then a QMC design sequence $(X_N)_{N\in A}$ for $\mathbb{H}^s(
\mathbb{S}^d )$, $s > \frac{d}{2}$, is characterised by
\begin{equation} \label{eq:QMC.design.sequence.characterization}
\left| \numint[X_{N}]( f ) - \xctint( f ) \right| \leq 
\frac{c_{s,d}}{N^{\frac{s}{d}}} \, \| f \|_{\mathbb{H}^{s}} \quad 
\text{for all $f \in \mathbb{H}^{s}( \mathbb{S}^d )$.} 
\end{equation}
We note that the order of $N$ cannot be improved
\cite[Thm.~3]{Brauchart_Saff_Sloan+2014:qmc-designs}. It is shown in
\cite[Thm.~4]{Brauchart_Saff_Sloan+2014:qmc-designs} that a QMC design
sequence for $\mathbb{H}^s( \mathbb{S}^d )$, $s > \frac{d}{2}$, is
also a QMC design sequence for $\mathbb{H}^{s^\prime}( \mathbb{S}^d )$
for all $\frac{d}{2} < s^\prime < s$.
A fundamental unresolved problem is to determine the supremum $s^*$
(called the strength of the sequence) of those $s$ for which
\eqref{eq:QMC.design.sequence.characterization}
holds. 
We prove the following result.

\begin{theorem} \label{thm:QMC.design.sequences.are.hyperuniform} A
  QMC design sequence for $\mathbb{H}^{s}( \mathbb{S}^d )$ with $s
  \geq \frac{d + 1}{2}$ is hyperuniform for large caps, small caps,
  and caps at threshold order.
\end{theorem}

It is known \cite[Thm.~14]{Brauchart_Saff_Sloan+2014:qmc-designs} that
points that maximise their sum of mutual generalised Euclidean
distances,
$\sum_{j,k=1}^{N} \left|\mathbf{x}_j - \mathbf{x}_k \right|^{2\tau-d}$, 
form a QMC design sequence $(X_{N,\tau}^*)_{N\in \mathbb{N}}$ for
$\mathbb{H}^{\tau}( \mathbb{S}^d )$ if $\tau \in (\frac{d}{2},
\frac{d}{2} + 1 )$; i.e.,
\begin{equation*}
  \left| \numint[X_{N,\tau}^*]( f ) - \xctint( f ) \right| \leq 
\frac{c_{s,d}}{N^{\frac{s}{d}}} \, \| f \|_{\mathbb{H}^s} \qquad 
\text{for all $f \in \mathbb{H}^s( \mathbb{S}^d )$ 
and all $\frac{d}{2} < s \leq \tau$,} 
\end{equation*}
whereas a sequence $( Z_{N_t} )_{t \in \mathbb{N}}$ of spherical
$t$-designs with exactly the optimal order of points, $N_t \asymp
t^d$, has the remarkable property
\cite[Thm.~6]{Brauchart_Saff_Sloan+2014:qmc-designs} that
\begin{equation*}
\left| \numint[Z_{N_t}]( f ) - \xctint( f ) \right| \leq 
\frac{c_{s,d}}{N_t^{\frac{s}{d}}} \, \| f \|_{\mathbb{H}^s} \quad 
\text{for all $f \in \mathbb{H}^s( \mathbb{S}^d )$ and {\bf all} 
$s > \frac{d}{2}$}
\end{equation*}
and, therefore, $(Z_{N_t})_{t \in \mathbb{N}}$ is a QMC design
sequence for $\mathbb{H}^{s}( \mathbb{S}^d )$ for every $s >
\frac{d}{2}$.
As corollaries to
Theorem~\ref{thm:QMC.design.sequences.are.hyperuniform} we obtain

\begin{corollary} \label{cor:hyperuniform.sum.of.distance.point.sets}
  Let $\tau \in (\frac{d}{2}, \frac{d}{2} + 1 )$. A sequence
  $(X_{N,\tau}^*)_{N\in \mathbb{N}}$ of $N$-point sets that maximise
  the sum $\sum_{j,k=1}^{N} \left|\mathbf{x}_j - \mathbf{x}_k
  \right|^{2\tau-d}$ is hyperuniform for large caps, small caps, and
  caps at threshold order.
\end{corollary}

\begin{corollary}
  A sequence $(X_N)_{N\in A}$ of spherical $t(N)$-designs with $t(N)
  \geq c_d \, N^{\frac1d}$, $N \in A$, for some $c_d>0$ is
  hyperuniform for large caps, small caps, and caps at threshold
  order.
\end{corollary}

The Sobolev space $\mathbb{H}^s(\mathbb{S}^d)$ consists of
$\mathbb{L}_2$-functions on $\mathbb{S}^d$ with finite Sobolev norm
\begin{equation*}
\| f \|_{\mathbb{H}^s} := \sqrt{\langle f, f
  \rangle_{\mathbb{H}^s(\mathbb{S}^d)}} 
= \sqrt{\sum_{n=0}^\infty 
\frac{\left\| Y_n[f] \right\|_{L^2( \mathbb{S}^d )}^2}{b_n(s)}},
\end{equation*}
where $Y_n[f]$, $n \in \mathbb{N}_0$, are the projections
\begin{equation*}
  Y_n[f]( \mathbf{x} ) := \int_{\mathbb{S}^d} Z(d, n) P_n^{(d)}( \langle \mathbf{x}, \mathbf{y} \rangle ) \, f( \mathbf{y} ) \, \dd \sigma( \mathbf{y} ), \qquad \mathbf{x} \in \mathbb{S}^d,
\end{equation*}
and $( b_n(s) )_{n \in \mathbb{N}_0}$ can be any fixed sequence of
positive real numbers satisfying
\begin{equation} \label{eq:sequenceassumption}
b_n(s) \asymp \left( 1 + n \right)^{-2s}. 
\end{equation}
Since the point-evaluation functional is a bounded operator on
$\mathbb{H}^s(\mathbb{S}^d)$ whenever $s>\frac{d}{2}$, the Riesz
representation theorem assures the existence of a reproducing kernel
for $\mathbb{H}^s(\mathbb{S}^d)$. It can be readily verified that the
zonal kernel
\begin{equation*}
K^{(s)}(\mathbf{x},\mathbf{y}) = 
\sum_{n=0}^\infty b_n(s) Z(d,n) P_n^{(d)}(\langle\mathbf{x},\mathbf{y}\rangle)
\end{equation*}
has the reproducing kernel properties
\begin{equation*}
  K^{(s)}({\mathbf \cdot},{\mathbf x}) \in \mathbb{H}^s(\mathbb{S}^d),
  \quad 
\mathbf{x} \in \mathbb{S}^d, \qquad 
\langle f, K^{(s)}({\mathbf \cdot},{\mathbf x})
\rangle_{\mathbb{H}^s(\mathbb{S}^d)} = 
f({\mathbf x}), \quad \mathbf{x} \in \mathbb{S}^d, 
f \in \mathbb{H}^s(\mathbb{S}^d).
\end{equation*}
Thus, reproducing kernel Hilbert space techniques (see \cite{Hi1998}
for the case of the unit cube) provide the means to compute the
worst-case error.
Standard arguments (see \cite{Brauchart_Saff_Sloan+2014:qmc-designs}) yield
\begin{equation} \label{eq:wce2}
\left[ \WCE( \numint[X_{N}]; \mathbb{H}^s( \mathbb{S}^d ) ) \right]^2
= 
\frac{1}{N^2} \sum_{i,j=1}^{N} \sum_{n=1}^\infty b_n(s) Z(d,n) 
P_n^{(d)}( \langle \mathbf{x}_i, \mathbf{x}_j \rangle).
\end{equation}

We exploit the flexibility in the choice of the sequence $( b_n(s)
)_{n \in \mathbb{N}_0}$ defining reproducing kernel, Sobolev norm, and
worst-case error to connect the Laplace-Fourier expansion of the
number variance given in \eqref{eq:var-harmonic} with an appropriately
chosen worst-case error.

\begin{lemma} \label{lem:number.variance.wce}
The number variance satisfies 
\begin{equation}
V(X_N, \phi) \ll \left( \sin \phi \right)^{d-1} N^2 
\left[ \WCE( \numint[X_{N}]; \mathbb{H}^{\frac{d+1}{2}}( \mathbb{S}^d))\right]^2
\end{equation}
for any $N$-point set $X_N \subset \mathbb{S}^d$ and opening angle
${\phi \in (0, \frac{\pi}{2} )}$.
\end{lemma}

\begin{proof}
  Using the estimate \eqref{eq:Gegenbauer}, the coefficients in
  \eqref{eq:var-harmonic} satisfy the relation
\begin{equation*}
a_n( \phi )^2 \ll \frac{\left( \sin \phi \right)^{d-1}}{n^{d+1}}
\asymp 
\frac{\left( \sin \phi \right)^{d-1}}{( 1 + n )^{d+1}}.
\end{equation*}
The positive definiteness of the kernel function \eqref{eq:g_phi} yields
\begin{equation*}
V(X_N,\phi) \ll \left( \sin \phi \right)^{d-1}
\sum_{i,j=1}^N\sum_{n=1}^\infty 
\left( 1 + n \right)^{-(d+1)} Z(d,n)
  P_n^{(d)}(\langle\mathbf{x}_i,\mathbf{x}_j\rangle).
\end{equation*}
Comparison with \eqref{eq:wce2} while taking into account
\eqref{eq:sequenceassumption} gives the result.
\end{proof}

\begin{remark}
  It is interesting that the Sobolev space
  $\mathbb{H}^{\frac{d+1}{2}}( \mathbb{S}^d )$ plays such a special
  role: When endowed with the reproducing kernel $1 -
  \frac{\gamma_d}{d} | \mathbf{x} - \mathbf{y} |$ for $\mathbf{x},
  \mathbf{y} \in \mathbb{S}^d$, the worst-case error satisfies the
  following invariance principle
  \cite{Brauchart_Dick2013:stolarskys_invariance} (see
  \cite{Sloan_Womersley2004:extremal_sphere,BiLa2017,Brauchart_Dick2013:characterization_sobolev_spaces,BrDiFa2015}
  for generalisations)
\begin{equation*}
\frac{1}{N^2} \sum_{j,k=1}^{N} \left|\mathbf{x}_j - \mathbf{x}_k
\right| + 
\frac{d}{\gamma_d} \left[\WCE(\numint[X_{N}]; 
\mathbb{H}^s( \mathbb{S}^d ) ) \right]^{2} = 
\int_{\mathbb{S}^d} \int_{\mathbb{S}^d}
\left|\mathbf{x}-\mathbf{y}\right| 
\dd \sigma(\mathbf{x}) \,\dd \sigma(\mathbf{y}),
\end{equation*}
which is equivalent \cite{Brauchart_Dick2013:stolarskys_invariance}
with Stolarsky's invariance
principle~\cite{Stolarsky1973:sums_distances_ii}, where the place of
the worst-case error is taken by the $L^2$-discrepancy given in
\eqref{eq:discr-L2}. Hence, an $N$-point system with maximal sum of
all mutual Euclidean distances is both a node set for a QMC method
that minimises the worst-case error in the above setting and a point
set with smallest possible $L^2$-discrepancy among all $N$-point sets
on $\mathbb{S}^d$.
A sequence of such maximal sum-of-distance $N$-point sets as $N \to
\infty$ is a QMC design sequence for at least
$\mathbb{H}^{\frac{d+1}{2}}( \mathbb{S}^d )$ with yet unknown strength
$s^*$ and thus is hyperuniform for large caps, small caps, and caps at
threshold order
(Corollary~\ref{cor:hyperuniform.sum.of.distance.point.sets}).
For the Weyl sums we get (cf. Remark~\ref{rmk:convergence.Weyl.sums})
for every fixed $n \in \mathbb{N}$ the limit relation
\begin{equation*}
\lim_{N \to \infty} N^{-1 + \frac{1}{d} - \eps} \sum_{i,j=1}^{N}
P_n^{(d)}( \langle \mathbf{x}_{i}, \mathbf{x}_{j} \rangle) = 0 
\qquad \text{for all sufficiently small $\eps > 0$.} 
\end{equation*}
\end{remark}

We are ready to prove Theorem~\ref{thm:QMC.design.sequences.are.hyperuniform}.

\begin{proof}[Proof of Theorem~\ref{thm:QMC.design.sequences.are.hyperuniform}]
  Let $(X_N)_{N \in A}$ be a QMC design sequence for $\mathbb{H}^s(
  \mathbb{S}^d )$ with $s \geq \frac{d+1}{2}$. Then, by
  \cite[Theorem~4]{Brauchart_Saff_Sloan+2014:qmc-designs}, it is also
  a QMC design sequence for $\mathbb{H}^{\frac{d+1}{2}}( \mathbb{S}^d
  )$; i.e.,
\begin{equation*}
\left[ \WCE( \numint[X_{N}]; \mathbb{H}^{\frac{d+1}{2}}( \mathbb{S}^d
  ) ) \right]^2 \leq c \, N^{-\frac{d+1}{d}}
\end{equation*}
for some constant $c > 0$. By Lemma~\ref{lem:number.variance.wce},
\begin{equation} \label{eq:number.variance.wce.estimate} V(X_N, \phi)
  \ll \left( \sin \phi \right)^{d-1} N^2 \, N^{-\frac{d+1}{d}} =
  \left( \sin \phi \right)^{d-1} N^{1-\frac{1}{d}}
\end{equation}
for the $X_N$ of the QMC design sequence $(X_N)_{N \in A}$ and any
opening angle ${\phi \in (0, \frac{\pi}{2} )}$.

{\bf (i) Large cap regime:} Let $\phi \in ( 0, \frac{\pi}{2} )$. Then,
by \eqref{eq:number.variance.wce.estimate},
\begin{equation*}
\begin{split}
  \frac{1}{N} \, V(X_N, \phi) \ll \left( \sin \phi \right)^{d-1}
  N^{-\frac{1}{d}} \to 0 \qquad \text{as $N \to \infty$.}
\end{split}
\end{equation*}
Consequently, for all $\phi \in ( 0, \frac{\pi}{2} )$,
\begin{equation*}
V(X_N, \phi) = o( N ) \qquad \text{as $N \to \infty$,}
\end{equation*}
and $(X_N)_{N \in A}$ is hyperuniform for large caps.

{\bf (ii) Small cap regime:} 
Let $(\phi_N)_{N \in A}$ be a sequence of radii satisfying $\phi_N \to
0$ and $N \sigma( C( \cdot, \phi_N ) ) \to \infty$ as $N \to
\infty$. Then, by \eqref{eq:number.variance.wce.estimate} and
\eqref{eq:normalised.surface.cap},
\begin{equation*}
  \frac{V(X_N, \phi_N)}{N \sigma( C( \cdot, \phi_N ) )} \ll 
\frac{\left( N \left( \sin \phi_N \right)^{d} \right)^{\frac{d-1}{d}}}
{N \left( \sin \phi_N \right)^{d}} 
\ll \Big( N \, \sigma( C( \cdot, \phi_N ) ) \Big)^{-\frac{1}{d}} \to 0 
\qquad \text{as $N \to \infty$;}
\end{equation*}
thus, $(X_N)_{N \in A}$ is hyperuniform for small caps.

{\bf (iii) threshold regime:} Suppose $( \phi_N )_{N \in A}$, $\phi_N
\in (0, \frac{\pi}{2} )$ such that $\phi_N = t \, N^{-\frac{1}{d}}$,
$t > 0$.
By \eqref{eq:number.variance.wce.estimate}, 
\begin{align*}
V(X_{N}, \phi_{N}) 
&\ll \left( \frac{\sin \phi_{N}}{\phi_{N}} \right)^{d-1} 
\left( \phi_{N} \right)^{d-1} N^{1-\frac{1}{d}} \\
&= \left( t \, N^{-\frac{1}{d}} \right)^{d-1} N^{1-\frac{1}{d}} \\
&= t^{d-1} \qquad \text{as $N \to \infty$.}
\end{align*}
The implied constant does not depend on $t$. Since $t> 0$ was
arbitrary,
\begin{equation*}
\limsup_{\substack{N\to\infty \\ N \in A}} V\Big(X_N, t \,
N^{-\frac{1}{d}} \Big) 
= \mathcal{O}( t^{d-1} ) \qquad \text{as $t \to \infty$}
\end{equation*}
and $(X_N)_{N \in A}$ is hyperuniform for caps at threshold order.
\end{proof}

\begin{remark} \label{rmk:convergence.Weyl.sums} Any QMC design
  sequence $(X_N)_{N \in A}$ for $\mathbb{H}^{s}( \mathbb{S}^d )$, $s
  \geq \frac{d+1}{2}$, is hyperuniform for large caps and thus, by
  Theorem~\ref{thm:limit.hyperuniformity.large.caps}, satisfies the
  property
\begin{equation*}
\lim_{\substack{N\to\infty \\ N \in A}} 
\frac{1}{N} \sum_{i,j=1}^N P_n^{(d)}( \langle \mathbf{x}_{i},
\mathbf{x}_{j} \rangle ) = 0 \qquad \text{for every $n \in
  \mathbb{N}$.}
\end{equation*}
As QMC design sequences are characterised by a bound on the worst-case
error, we can use such bounds to quantify the convergence of Weyl sums
along the sequence.
More generally, let $(X_N)_{N \in A}$ be a sequence of $N$-point sets
on $\mathbb{S}^d$ with finite strength $s^* > \frac{d}{2}$; i.e.,
\begin{equation} \label{eq:strength.estimate}
\left| \numint[X_{N}]( f ) - \xctint( f ) \right| \leq 
\frac{c_{s,d}}{N^{\frac{s}{d}}} \, \| f \|_{\mathbb{H}^{s}} 
\quad \text{for all $f \in \mathbb{H}^{s}( \mathbb{S}^d )$ 
and all $\frac{d}{2} < s < s^*$} 
\end{equation}
and this relation fails if $s > s^*$. Then for every fixed $n \in
\mathbb{N}$, we get the limit relation
\begin{equation} \label{eq:Weyl.sum.estimate.non-varying.n}
\lim_{\substack{N \to \infty \\ N \in A}} N^{-1 + \frac{2s^* - d \, (
    1 + \eps )}{d}} \sum_{i,j=1}^{N}  P_n^{(d)}( \langle
\mathbf{x}_{i}, \mathbf{x}_{j} \rangle) = 0 \qquad \text{for all
  sufficiently small $\eps > 0$,}
\end{equation}
which follows from the following estimate that combines
\eqref{eq:wce2} and \eqref{eq:strength.estimate}: for $s = s^* -
\frac{d}{4} \, \eps$,
\begin{equation*}
0 \leq \sum_{n=1}^\infty b_n(s) \, Z(d,n) \, N^{\frac{2s^*}{d} - 2 -
  \eps} \sum_{i,j=1}^{N}  P_n^{(d)}( \langle \mathbf{x}_{i},
\mathbf{x}_{j} \rangle) \leq \frac{c_{s,d}^2}{N^{\frac{\eps}{2}}}
\end{equation*}
holds for all sufficiently small but fixed $\eps > 0$. 
Thus, the critical exponent 
\begin{equation*}
-1 + \frac{2s^*-d}{d}
\end{equation*}
of a given sequence of $N$-point sets is limited by its strength. 
The connection between number variance and worst-case error given in
Lemma~\ref{lem:number.variance.wce} indicates that sequences with
strength $s^* \geq \frac{d+1}{2}$, where the critical exponent
satisfies
\begin{equation*}
-1 + \frac{2s^*-d}{d} \geq -1 + \frac{1}{d},
\end{equation*}
are of particular interest.

We conclude this remark by considering the case when $n$ in
\eqref{eq:Weyl.sum.estimate.non-varying.n} is not fixed. Assume $n
\leq c \, N^{\frac{\alpha}{d}} \, \Psi( N )$, $N \in A$, for some $c_d
> 0$, $\alpha \in \mathbb{R}$, and $\Psi$ such that $\Psi( N ) \to 0$
and $N^{\frac{\alpha}{d}} \, \Psi( N ) \to \infty$ as $N \to \infty$.
Then, by \eqref{eq:sequenceassumption} and $Z(d, n ) \asymp n^{d-1}$,
we get for $\frac{d}{2} < s < s^*$,
\begin{equation*}
b_n( s ) \, Z(d,n) \, N^{\frac{2s}{d} - 2} \asymp 
n^{-( 2s - d + 1 )} \, N^{-1 + \frac{2s - d}{d}} \gg  
N^{-1 - \frac{1}{d} + \frac{1-\alpha}{d} \, ( 2s - d + 1 )}\left(\Psi(
  N )\right)^{-( 2s - d + 1 )}
\end{equation*}
and thus
\begin{equation} \label{eq:Weyl.sum.estimate.varying.n} N^{-1 -
    \frac{1}{d} + \frac{1-\alpha}{d} \, ( 2s - d + 1 )} \,
  \sum_{i,j=1}^{N} P_{n}^{(d)}( \langle \mathbf{x}_{i}, \mathbf{x}_{j}
  \rangle) = \mathcal{O}( \left( \Psi( N ) \right)^{2s - d + 1} )
\end{equation}
uniformly in $n \leq c \, N^{\frac{\alpha}{d}} \, \Psi( N )$, $N \in
A$. (The implied constant is independent of $n$ and~$N$.)
The function $\Psi( N )$ may tend to zero arbitrarily slow as $N \to
\infty$.
The value of $\alpha$ in the power $N^{\frac{\alpha}{d}}$ determines
three regimes of growth of the bound of $n$.
If the bound for $n$ does not grow too fast (i.e., $\alpha \in (0,
1)$), the parameter $s$ effectively enlarges the exponent of~$N$. A
sequence of higher strength $s^*$ allows for larger powers of $N$. The
effective exponent is then strictly larger than
$-1-\frac{\alpha}{d}$. For the critical value $\alpha = 1$, the
exponent of $N$ does not depend on $s$ at all. It is always
$-1-\frac{1}{d}$. If the bound for $n$ grows too fast (i.e., $\alpha >
1$), then the effective exponent of $N$ is strictly smaller than
$-1-\frac{\alpha}{d}$.
\end{remark}

\begin{remark}
A sequence $(Z_N)_{N \in A}$ with infinite strength has the property
\begin{equation*}
\lim_{\substack{N \to \infty \\ N \in A}} N^{-1+\beta} \sum_{i,j=1}^{N}  P_n^{(d)}( \langle \mathbf{z}_{i}, \mathbf{z}_{j} \rangle) = 0 \qquad \text{for every fixed $\beta > 0$ and fixed $n \in \mathbb{N}$,}
\end{equation*}
while relation \eqref{eq:Weyl.sum.estimate.varying.n}, in particular, implies that 
\begin{equation*}
\lim_{\substack{N \to \infty \\ N \in A}} N^{-1+\beta} \sum_{i,j=1}^{N}  P_n^{(d)}( \langle \mathbf{z}_{i}, \mathbf{z}_{j} \rangle) = 0 \qquad \text{for every fixed $\beta > 0$}
\end{equation*}
uniformly in $n \leq c \, N^{\frac{\alpha}{d}} \, \Psi( N )$, $N \in A$, if $0 < \alpha < 1$ and under the assumptions $\Psi( N ) = o( 1 )$ and $N^{\frac{\alpha}{d}} \, \Psi( N ) \to \infty$ as $N \to \infty$.

So far (see \cite{Brauchart_Saff_Sloan+2014:qmc-designs}), the only example of such sequences are sequences of spherical $t(N)$-designs with $t(N) \asymp N^{\frac{1}{d}}$.

\end{remark}

\begin{remark}
  As indicated in Section \ref{sec:sphere}, we restricted this study
  to the sphere for ease of computation. Most of the results would
  extend \emph{mutatis mutandis} to other homogeneous spaces like the
  torus or the projective plane. We expect that the definition of
  hyperuniformity would carry over to compact Riemannian manifolds
  with considerably more effort and technicalities in the harmonic
  analysis.
\end{remark}

\begin{ackno}
  This work was initiated during the programme ``Minimal Energy Point Sets,
  Lattices, and Designs'' held at the Erwin Schr\"odinger Institute in Vienna
  in Fall 2014. There Salvatore Torquato asked how to properly define a notion
  of hyperuniformity on compact sets. This paper is an attempt to answer that
  question.
The authors are grateful to two anonymous referees for their
  many valuable remarks that helped improving the presentation of the paper.
\end{ackno}

\providecommand{\bysame}{\leavevmode\hbox to3em{\hrulefill}\thinspace}
\providecommand{\MR}{\relax\ifhmode\unskip\space\fi MR }
\providecommand{\MRhref}[2]{%
  \href{http://www.ams.org/mathscinet-getitem?mr=#1}{#2}
}
\providecommand{\href}[2]{#2}

\end{document}